\documentclass[12pt,a4paper,reqno]{amsart}

\DeclareFontFamily{OT1}{rsfs}{}
\DeclareFontShape{OT1}{rsfs}{n}{it}{<-> rsfs10}{}
\DeclareMathAlphabet{\mathscr}{OT1}{rsfs}{n}{it}

\theoremstyle{plain}

\usepackage{amssymb,amsmath,graphicx,amsthm,mathrsfs,amsbsy}
\usepackage{todonotes}
\usepackage{amscd,verbatim,mathtools}
\usepackage{hyperref}
\usepackage{psfrag,epsfig}
\usepackage[ruled,vlined]{algorithm2e}
\usepackage{enumerate}
\usepackage{bm}
\usepackage{float}
\usepackage[caption = false]{subfig}

\newtheorem{theorem}{Theorem}[section]
\newtheorem{proposition}[theorem]{Proposition}
\newtheorem{example}[theorem]{Example}

\def\c1{\mathbf 1}

 \title[Hamilton Powers of Eulerian Digraphs]{Hamilton Powers of Eulerian Digraphs}
\author{Enrico Celestino Colón}
\address{Department of Mathematics, Harvard University, Cambridge, MA USA.}
\email{ecolon@math.harvard.edu}
\author{John Urschel}
\address{Department of Mathematics, Massachusetts Institute of Technology, Cambridge, MA USA.}
\email{urschel@mit.edu}
\begin{document}

\maketitle

\begin{abstract}
In this note, we prove that the $\lceil \tfrac{1}{2} \sqrt{n} \log_2^2 n \rceil^{th}$ power of a connected $n$-vertex Eulerian digraph is Hamiltonian, and provide an infinite family of digraphs for which the $\lfloor \sqrt{n}/2 \rfloor^{th}$ power is not.
 \end{abstract}

\section{Preliminaries}

The $k^{th}$ power of a (directed or undirected) graph $G$, denoted $G^k$, is the graph on the vertices of $G$ in which there is an edge from a vertex $u$ to a vertex $v$ if there exists a $uv$-path in $G$ of length at most $k$.
It is well-known that the cube of any connected undirected graph is Hamiltonian (see \cite{karaganis1968cube,sekanina1960ordering}, also \cite[Ex 10-14]{diestel2005graph}). In 1974, Fleischner proved that the square of any two-connected undirected graph is Hamiltonian, solving the Plummer-Nash-Williams conjecture \cite{fleischner1974square} (see \cite{georgakopoulos2009short} for a much simpler proof). Unfortunately, strongly-connected directed graphs (digraphs) may require the $\lceil n/2 \rceil^{th}$ power to be Hamiltonian; even $k$-strong connectedness is only sufficient for guaranteeing that the $\lceil n/(2k) \rceil^{th}$ power is Hamiltonian \cite{schaar1997upper}. For a general survey on Hamilton cycles in digraphs, we refer the reader to \cite{kuhn2012survey}. Interestingly, results for Eulerian digraphs are not nearly so bleak\footnote{The first notable example of a class of digraphs requiring a ``non-trivial" (say, $o(n)$) Hamiltonicity exponent are cacti, see \cite{schaar1994remarks} for details.}. Through the study of minimally Eulerian digraphs (connected Eulerian digraphs with no proper connected Eulerian subgraph), we prove that

\begin{theorem} \label{thm:main}	The $\lceil \tfrac{1}{2} \, \sqrt{n} \log_2^2 n \rceil^{th}$ power of any $n$-vertex connected Eulerian digraph is Hamiltonian.
\end{theorem}

In fact, we prove an even stronger result (in Theorem \ref{thm:main_detail}) that, given a minimally Eulerian digraph $G = (V,A)$, specifies an ordering $v_1,...,v_n$ of $V$ and an edge-disjoint directed path (dipath) decomposition $P_1,...,P_n$ of $G$, such that each $P_i$ is a $v_i v_{i+1}$-dipath ($v_{n+1}:=v_1$) of length at most $\lceil \tfrac{1}{2} \, \sqrt{n} \log_2^2 n \rceil$. In addition, we provide an infinite family of minimally Eulerian digraphs for which the $\lfloor \sqrt{n}/2 \rfloor^{th}$ power is not Hamiltonian (Example \ref{ex:graph}). For details regarding the importance of minimally Eulerian digraphs and their connection to the traveling salesman problem, we refer the reader to \cite{doi:10.1137/0213007,papadimitriou1981minimal}.

\subsection{Notation, Definitions, and Basic Results}

Let $G = (V,A)$ be a simple digraph. If $G$ contains a spanning directed cycle (dicycle), then $G$ is Hamiltonian. If $G$ contains an Euler circuit (a circuit containing every edge), then $G$ is Eulerian. If $G$ is connected, this is equivalent to the condition that, for every vertex $v \in V$, the indegree $d^-(v)$ equals the outdegree $d^+(v)$. If $G$ is a connected Eulerian digraph and contains no proper connected Eulerian subgraph on the vertices of $G$, then $G$ is minimally Eulerian; equivalently, a connected Eulerian digraph $G$ is minimally Eulerian if, for any dicycle $C$ of $G$, the graph $G-C:= (V,A - A(C))$ is disconnected. If $G$ contains no dicycle, then $G$ is acyclic. For more details regarding graph theoretic definitions and notation, we refer the reader to \cite{bollobas1998modern}. Let 
 $$p_{\#}(G) := \frac{1}{2} \sum_{u \in V} | d^+(u) - d^-(u) |,$$
 a measure of how ``close" to Eulerian a digraph is, and a key ingredient in our proof. The quantity $p_{\#}(G)$ is exactly the minimal number of dipaths required in an edge-disjoint decomposition of $G$ into dipaths and dicycles. That $p_{\#}(G)$ dipaths are required follows immediately from the definition of $p_{\#}(G)$ above. That $p_{\#}(G)$ dipaths are sufficient follows from a simple greedy algorithm (iteratively perform walks from vertices $u$ with $d^+(u)>d^-(u)$, removing dicycles when they are formed, and only removing the dipath when a vertex $v$ with $d^+(v) = 0$ is reached). 
 The size of an acyclic digraph $G$ is immediately bounded above by $p_{\#}(G) \,(|V|-1)$, and an even tighter estimate can be obtained relatively quickly:
 
\begin{proposition} \label{prop:acyclic}
	 Let $G = (V,A)$ be an acyclic digraph. Then $|A| \le \sqrt{2 p_{\#}(G)} \, |V|$.
\end{proposition}

\begin{proof}
If $p_{\#}(G) =0,1,2$, the result follows immediately, as $|A| \le p_{\#}(G) (|V|-1)$. Now, let $p_{\#}(G) > 2$, $V = \{v_1,...,v_n\}$ be a topological sorting of $G$ (i.e., $v_iv_j\in A$ implies that $i<j$), $k \in \mathbb{N}$ be the smallest number such that $p_{\#}(G) \le {k \choose 2}$, $\ell = \lceil n/k \rceil$, and $V_i = \{v_{(i-1) k+1},..., v_{ik} \} $, $i = 1,...,\ell-1$, $V_\ell = \{v_{(\ell-1)k+1},...,v_n\}$. There are at most ${k \choose 2}$ edges within each of the subsets $V_i$, $i =1,...,\ell-1$, and at most ${n - k(\ell-1) \choose 2}$ within the subset $V_\ell$. Our digraph $G$ can be decomposed into $p_{\#}(G)$ edge-disjoint dipaths, and, by the topological sorting of $V$, each of the aforementioned $p_{\#}(G)$ dipaths has at most $\ell-1$ edges between the subsets $V_1,...,V_\ell$. Therefore, there are at most $(\ell -1) p_{\#}(G)$ total edges between the subsets $V_1,...,V_\ell$. Combining these estimates gives
	$$|A| \le \big(\ell -1 \big) \big[ \textstyle{{k \choose 2}} + p_{\#}(G) \big] + {n - k(\ell -1) \choose 2}.$$
Dividing by $\sqrt{p_{\#}(G)} \, n$, we have
$$	\frac{|A|}{\sqrt{p_{\#}(G)} \, n}\le \frac{\ell -1 }{n} \bigg( \frac{{k \choose 2}}{\sqrt{p_{\#}(G)}}+\sqrt{p_{\#}(G)}\bigg)
	+\frac{{n - k(\ell -1) \choose 2}}{\sqrt{p_{\#}(G)} n}. $$
 
 The right hand side is convex w.r.t. $p_{\#}(G)$ and maximized when $p_{\#}(G)$ is as small as possible. We note that, by the definition of $k$, $p_{\#}(G) > {k-1 \choose 2}$. So the right hand side can be bounded above by replacing $p_{\#}(G)$ by ${k-1 \choose 2}$, giving
$$\frac{|A|}{\sqrt{p_{\#}(G)} \, n}< \frac{\ell -1 }{n} \, \frac{(k-1)^2}{{k-1 \choose 2}^{1/2}}
	+\frac{\big(n - k(\ell -1)\big) \big(n - k(\ell -1)-1\big)}{2{k-1 \choose 2}^{1/2} n}.$$

	The right hand side is a convex quadratic function in the term $\ell$ (treating $\ell$ as a variable independent of $n$ and $k$), and therefore achieves its maximum at one of the endpoints of the interval $[n/k,n/k+1]$. Setting $\ell = n/k$ gives 
    $$ \frac{\ell -1 }{n} \, \frac{(k-1)^2}{{k-1 \choose 2}^{1/2}}
	+\frac{\big(n - k(\ell -1)\big) \big(n - k(\ell -1)-1\big)}{2{k-1 \choose 2}^{1/2} n}= \frac{(k-1)^2}{k {k-1 \choose 2}^{1/2}}- \frac{k^2-3k+2}{2n{k-1 \choose 2}^{1/2}},$$     
and setting $\ell = n/k +1$ gives
    $$ \frac{\ell -1 }{n} \, \frac{(k-1)^2}{{k-1 \choose 2}^{1/2}}
    +\frac{\big(n - k(\ell -1)\big) \big(n - k(\ell -1)-1\big)}{2{k-1 \choose 2}^{1/2} n}= \frac{(k-1)^2}{k {k-1 \choose 2}^{1/2}}.$$
    
Noting that $k^2-3k+2\ge 0$ for all $k \in \mathbb{N}$, we conclude that the maximum over the interval $[n/k,n/k+1]$ is obtained at $\ell = n/k+1$. Replacing $\ell$ by $n/k+1$, we have
	$$	|A| < \frac{(k-1)^2}{k \textstyle{{k-1 \choose 2}}^{1/2} } \, \sqrt{p_{\#}(G)} \, n = \frac{(k-1)^{3/2}}{k(k-2)^{1/2}}\sqrt{2p_{\#}(G) }\, n \le \sqrt{2 p_{\#}(G)} \, n,$$
for $k \ge 3$ (recall, $p_{\#}(G) > 2$).
	\end{proof}

From Proposition \ref{prop:acyclic} we immediately obtain a bound (tight up to a multiplicative constant; see Example \ref{ex:graph}) on the maximum size of a minimally Eulerian digraph:

\begin{proposition} \label{prop:mineuleriansize}
	Let $G=(V,A)$ be a minimally Eulerian digraph. Then \newline
	$|A| \le \sqrt{2(|V|-1)} \; |V| + |V| -1$.
\end{proposition}

\begin{proof}
    $G$ is a connected Eulerian digraph, so it admits a rooted, directed subgraph $T$ of $G$ in which there is a unique path (in $T$) from the root to any other vertex of $G$. Every dicycle of $G$ must intersect an edge of $T$, as the removal of any dicycle from a minimally Eulerian graph disconnects it. Therefore, $G-T$ is acyclic, and by Proposition \ref{prop:acyclic}, $|A| \le |A(G-T)|+ |A(T)| \le \sqrt{2(|V|-1)} \; |V| + |V| -1$.
\end{proof}

\section{A Proof of Theorem \ref{thm:main} and a Lower Bound}

To prove Theorem \ref{thm:main}, we show an even stronger statement regarding minimally Eulerian digraphs.

\begin{theorem}\label{thm:main_detail}
    Let $G = (V,A)$, $|V|=n>1$, be a minimally Eulerian graph. Then there exists an ordering $v_1,...,v_n$ of $V$ and an $n$-dipath edge-disjoint decomposition $P_1,...,P_n$ of $G$ such that each $P_i$ is a $v_iv_{i+1}$-dipath ($v_{n+1}:=v_1$) of length at most $\lceil f(n) \, \sqrt{n}\, \rceil$, where
    $$f(n) =  (\log_2 n)^{\log_{3/2} 2 + o(1)} \le \frac{1}{2} \log^2_2 n.$$ 
\end{theorem}

\begin{proof}
We first show that there exists an ordering $v_1,...,v_n$ of $V(G)$ such that there is an $n$-dipath edge-disjoint decomposition $P_1,...,P_n$ of $G$ such that each $P_i$ is a $v_i v_{i+1}$-dipath. This ordering and decomposition can be constructed by picking a base vertex $v_1 \in V(G)$ and considering an Eulerian circuit $W$ of $G$ starting at $v_1$, ordering the remaining vertices based on the order of first appearance in this circuit, and taking each dipath $P_i$ to be the walk in $W$ between the first appearance of $v_i$ and the first appearance of $v_{i+1}$. As $G$ is minimally Eulerian, each such walk is a dipath. It suffices to consider $n \ge 6388$, as the length of a dipath is at most $n-1$ and $\lceil \tfrac{1}{2} \sqrt{n}  \log^2_2 n \, \rceil \ge n-1$ for $n=1,..., 6387$.

Let $v_1,...,v_n$ be an ordering of $V(G)$ and $P_1,...,P_n$ a decomposition of $G$ into edge-disjoint $v_i v_{i+1}$-dipaths $P_i$. We choose this ordering and decomposition so that the elements of the set $\{|A(P_1)|,...,|A(P_n)|\}$ are lexicographically minimized (i.e., minimizes the length of the longest dipath, minimizes the length of the $2^{nd}$ longest dipath conditional on the minimality of the longest dipath, etc). Let $\widehat P$ be the longest dipath in the set $\{P_1,...,P_n\}$, with length $|A(\widehat P)|=\alpha \sqrt{n}$ for some $\alpha \ge \tfrac{1}{2} \big[\log_2 n\big]^{\log_{3/2} 2}$. We aim to build a sequence of subgraphs $H_0(:=\widehat P)\subset H_1 \subset H_2\subset...$, bound the order of each subgraph from below using the lexicographic minimality of path lengths, and conclude that if $\alpha$ is too large then some $H_i$ contains too many vertices, thus producing an upper bound on $\alpha$.

Let $H_0 = \widehat P$. Let $H_{\ell}$, $\ell >0$, be the union of all $P_i$ satisfying both $|A(P_i)|\ge \alpha \sqrt{n}/2^{\ell}$ and  $\{v_i,v_{i+1}\}\cap V(H_{\ell-1})\neq \varnothing$. Let $n_\ell$, $m_\ell$, and $k_\ell$ be the number of vertices, edges, and dipaths $P_i$ in $H_\ell$. We have $n_0 = \alpha \sqrt{n} +1$, $m_0 = \alpha \sqrt{n}$, $k_0 =1$ and, by construction, $m_{\ell} \ge k_{\ell} m_0/2^{\ell}$ for all $\ell \ge 0$.

We may produce a lower bound for the size of each $H_\ell$ by our lexicographic minimality condition. We claim that every vertex of $H_\ell$ is either the start- or end-vertex of a dipath $P_i$ of length at least $m_0/2^{\ell+1}$. Suppose, to the contrary, that some $v_i \in V(H_\ell)$ satisfies $|A(P_{i-1})|,|A(P_{i})| <m_0/2^{\ell+1}$. Let $P_j$ be a dipath in $H_\ell$ containing $v_i$,
and let us denote the $v_{j}v_{i}$ (resp. $v_{i}v_{j+1}$) portion of this path by $P_j^1$ (resp. $P_j^2$).
By removing $P_i$, $P_{i+1}$, and $P_j$ from our set $\{P_1,...,P_n\}$ and replacing them with $P_j^1$, $P_j^2$, and $P_i \cup P_{i+1}$, we have replaced a path of length $|A(P_j)|$ ($|A(P_j)| \ge m_0/2^{\ell}$) with paths all of length strictly less than $|A(P_j)|$, a contradiction. Therefore, $k_{\ell+1} \ge n_{\ell}/2$ for all $\ell \ge 0$, as every vertex in $V(H_\ell)$ is the start- or end-vertex of a dipath $P_i$ in $H_{\ell+1}$, and each dipath $P_i$ has only one start- and one end-vertex.

The graph $H_\ell$ can be decomposed into the edge-disjoint union of two graphs $H_{\ell,a}$ and $H_{\ell,e}$, where $H_{\ell,a}$ is acyclic with $p_{\#}(H_{\ell,a}) \le k_\ell$ (as $H_{\ell}$ is the edge-disjoint union of $k_\ell$ paths) and $H_{\ell,e}$ is the vertex-disjoint union of minimally Eulerian graphs $H_{\ell,e}^{(1)},...,H_{\ell,e}^{(p_\ell)}$ for some $p_\ell$ (if the Eulerian graph $H_{\ell,e}^{(j)}$ is not minimal, neither is $G$). By Proposition \ref{prop:acyclic}, $H_{\ell,a}$ has at most $\sqrt{2 k_\ell} \, n_\ell$ edges. By Proposition \ref{prop:mineuleriansize}, $H_{\ell,e}$ has at most

$$\sum_{j=1}^{p_\ell} \left(\sqrt{2(n_\ell^{(j)}-1)}\, n_\ell^{(j)}+n_{\ell}^{(j)}-1\right) \leq \sqrt{2(n_\ell-1)} \, n_\ell +n_\ell-1$$
edges, where $n_\ell^{(j)}:=|V(H_{\ell,e}^{(j)})|$, $j=1,...,p_\ell$. Therefore,
\begin{equation*}
    m_\ell \le \sqrt{2 k_\ell} \, n_\ell +\sqrt{2(n_\ell-1)} \, n_\ell +n_\ell-1.
\end{equation*}
Combining this inequality with the bound $m_{\ell} \ge k_{\ell} m_0/2^{\ell}$, we have
\begin{equation}\label{ineq:Hsize}
    k_{\ell} m_0/2^{\ell} \le \sqrt{2 k_\ell} \, n_\ell +\sqrt{2(n_\ell-1)} \, n_\ell +n_\ell-1.
\end{equation}

Using Inequality (\ref{ineq:Hsize}), we produce a recursive lower bound on $n_\ell$ that gives an upper bound on $\alpha$. In particular, we aim to show that
\begin{equation}\label{ineq:recursive-bound}
    n_\ell \ge \bigg( \frac{ n_{\ell-1} m_0}{5 \times 2^\ell} \bigg)^{2/3} \quad \text{for all} \; \ell \le \log_2(5^2 \alpha).
\end{equation}

If $n_\ell \ge \sqrt{2 k_{\ell}} \, m_0/2^\ell$, then Inequality (\ref{ineq:recursive-bound}) immediately holds, as
$$n_\ell \ge \frac{\sqrt{2 k_{\ell}} \, m_0}{2^{\ell} } = \bigg[\bigg( \frac{ n_{\ell-1} m_0}{5 \times 2^\ell} \bigg)^2 \bigg(\frac{(2 k_\ell)^{3/2} n^{1/2}}{n_{\ell-1}^2} \bigg) \bigg(\frac{5^2 \, \alpha}{2^\ell} \bigg)   \bigg]^{1/3} \ge \bigg( \frac{ n_{\ell-1} m_0}{5 \times 2^\ell} \bigg)^{2/3}$$
for $\alpha \ge 2^\ell/5^2$. Now, suppose that $n_\ell < \sqrt{2 k_{\ell}} \, m_0/2^\ell$. Then $k_{\ell} m_0/2^{\ell} - \sqrt{2 k_\ell} n_\ell$ is monotonically increasing with respect to $k_\ell$. Combining this fact with the bound $k_\ell \ge n_{\ell-1}/2$ and Inequality (\ref{ineq:Hsize}), we obtain
$$n_{\ell-1} m_0/2^{\ell+1} - \sqrt{n_{\ell-1}} \, n_\ell  \le k_{\ell} m_0/2^{\ell} - \sqrt{2 k_\ell} \, n_\ell \le \sqrt{2(n_\ell-1)} \, n_\ell +n_\ell-1.$$
This implies that 
$$n_{\ell-1} m_0/2^{\ell+1} \le \sqrt{2(n_\ell-1)} \, n_\ell +  \sqrt{n_{\ell-1}} \, n_\ell + n_\ell-1 < \frac{5}{2} n_\ell^{3/2},$$
for $n \ge 6388$, as $n_\ell \ge n_0 = \alpha \sqrt{n} +1$, and so the claim holds in this case as well.

Using the initial bound $n_0 >m_0$ and Inequality (\ref{ineq:recursive-bound}), we obtain
\begin{align*}
    n \ge n_\ell &\ge n_0^{(2/3)^\ell} \prod_{i=1}^\ell \bigg( \frac{m_0}{5\times 2^{\ell+1-i}} \bigg)^{(2/3)^i} \\ &= \frac{ n_0^{(2/3)^\ell}}{2^{2\ell}} \bigg(\frac{16m_0^2}{25}\bigg)^{1-(2/3)^\ell} \\&> \frac{16 \, m_0^{2-(2/3)^\ell}}{25 \times 2^{2\ell}} \\&=  \frac{16 \, \alpha^{2-(2/3)^\ell} n^{1-\tfrac{1}{2}(2/3)^\ell}}{25 \times 2^{2\ell}} 
\end{align*}
for $\ell \le \log_2(5^2 \alpha)$. Taking the logarithm of both sides, we obtain the inequality
\begin{equation}\label{ineq:ref}
\log_2 \alpha < \frac{1}{2-(2/3)^\ell} \big(  \log_2(25/16) + 2 \ell + \tfrac{1}{2}(2/3)^\ell \log_2 n \big).
\end{equation}
Setting
$\ell = \big\lceil \log_{3/2} \big(\tfrac{3}{11}  \log_2 n \big) \big\rceil$, we have $\ell < \log_2(5^2 \alpha)$, as
\begin{align*}\log_{3/2} \big(\tfrac{3}{11}  \log_2 n \big) +1 &= (\log_{3/2} 2) \log_2(\log_2 n) + \log_{3/2}(3/11) +1 \\ &<  (\log_{3/2} 2) \log_2(\log_2 n) + 2\log_2(5) -1 \\ &=\log_2\big(\tfrac{5^2}{2} \log_2^{\log_{3/2} 2} n\big) .
\end{align*}
For $\ell = \big\lceil \log_{3/2} \big(\tfrac{3}{11}  \log_2 n \big) \big\rceil$, Inequality \ref{ineq:ref} implies that
\begin{align*}
    \log_2 \alpha &<  \frac{ \log_2(25/16) + 2 \big[\log_{3/2} \big(\tfrac{3}{11}  \log_2 n \big)+1 \big] + \tfrac{1}{2}(2/3)^{ \log_{3/2} \big(\tfrac{3}{11}  \log_2 n \big)} \log_2 n}{2-(2/3)^{ \log_{3/2} \big(\tfrac{3}{11}  \log_2 n \big)}} \\
    &=\frac{1}{1 - \frac{11}{6\log_2 n}} \big[ \log_2(5/2) + \log_{3/2} \big( \tfrac{3}{11} \log_2 n \big)+\tfrac{11}{12} \big].
\end{align*}
Taking the (base two) exponential of both sides, we obtain
$$ \alpha < 2^{\frac{\log_2(5/2) - \log_{3/2}(11/3) +11/12}{1 - 11/(6 \log_2 6388)}} \, \big[ \log_2 n \big]^{\frac{\log_{3/2} 2}{1 - 11/6\log_2 6388}} \le .46 \big[\log_2 n \big]^{1.9995}.$$
This completes the proof.
\end{proof}

\begin{figure}
    \centering
    \includegraphics[width = 4.5 in]{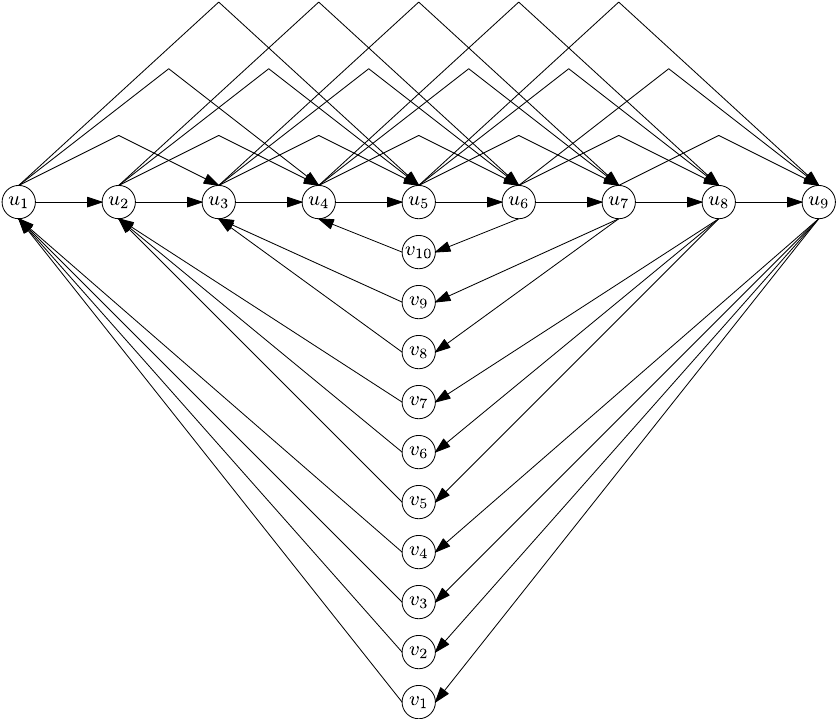}
    \caption{The minimally Eulerian graph $G_k$ from Example \ref{ex:graph} for $k = 4$.}
     \label{fig:small-graph}
\end{figure}
	
Finally, we give the following infinite class of digraphs to illustrate that Theorem \ref{thm:main} is tight up to a logarithmic factor.

\begin{example}\label{ex:graph}
Let $G_k = (V_k, A_k)$, $k \in \mathbb{N}$, $k \ge 4$, where $V_k = \{u_1,...,u_{\ell-1},v_1,...,v_{\ell}\}$, $\ell:=k(k+1)/2$, and
$u_iu_j \in A_k$ for $0<j-i \le k$, and $u_{\ell - \phi(i)} v_i,v_i u_{\phi(i)}\in A_k$ for all $i = 1,...,\ell$, where $\phi(i)$ is the smallest number $p \in \mathbb{N}$ such that $\sum_{j=1}^p(k+1-j) \ge i$. This digraph is minimally Eulerian, as every dicycle contains some vertex $v_i$ and $d^+(v_i) = d^-(v_i) = 1$ for all $i$. There are $n=k^2+k-1$ vertices and $k(k^2+2k-1)/2$ edges (i.e., about $n^{3/2}/2$). The distance between any pair $v_i,v_j$ in the graph is at least 
$\lceil (\ell+1)/k \rceil = \lceil k/2 \rceil +1 \ge \lfloor \sqrt{n}/2 \rfloor + 1$. In any Hamiltonian dicycle of a power of $G_k$, some pair $v_i,v_j$ must be adjacent, and so at least the $[\lfloor \sqrt{n}/2 \rfloor + 1]^{th}$ power is required. See Figure \ref{fig:small-graph} for a visual example for $k = 4$.
\end{example}

\section*{Acknowledgements}
This manuscript is the result of an undergraduate research project supported by Professor Michel Goemans through the MIT undergraduate research opportunities program (UROP), and we are thankful for his support. This project is also supported by the Lord Foundation through the MIT UROP program, and we are thankful for their support. This project is also supported by the MIT Office of Experiential Learning, and we are thankful for their support. This material is based upon work supported by the Institute for Advanced Study and the National Science Foundation under Grant No. DMS-1926686. The authors are grateful to Louisa Thomas for improving the style of presentation.

{ \small 
	\bibliographystyle{plain}
	\bibliography{main} }

\begin{thebibliography}{10}

\bibitem{bollobas1998modern}
B{\'e}la Bollob{\'a}s.
\newblock {\em Modern graph theory}, volume 184 of {\em Graduate texts in mathematics}.
\newblock Springer Science \& Business Media, 1998.

\bibitem{doi:10.1137/0213007}
Stavros~S Cosmadakis and Christos~H Papadimitriou.
\newblock The traveling salesman problem with many visits to few cities.
\newblock {\em SIAM Journal on Computing}, 13(1):99--108, 1984.

\bibitem{diestel2005graph}
Reinhard Diestel.
\newblock {\em Graph theory}, volume 173 of {\em Graduate texts in mathematics}.
\newblock Springer Science \& Business Media, third edition, 2005.

\bibitem{fleischner1974square}
Herbert Fleischner.
\newblock The square of every two-connected graph is {H}amiltonian.
\newblock {\em Journal of Combinatorial Theory, Series B}, 16(1):29--34, 1974.

\bibitem{georgakopoulos2009short}
Agelos Georgakopoulos.
\newblock A short proof of {F}leischner’s theorem.
\newblock {\em Discrete Mathematics}, 309(23-24):6632--6634, 2009.

\bibitem{karaganis1968cube}
Jerome~J Karaganis.
\newblock On the cube of a graph.
\newblock {\em Canadian Mathematical Bulletin}, 11(2):295--296, 1968.

\bibitem{kuhn2012survey}
Daniela K{\"u}hn and Deryk Osthus.
\newblock A survey on {H}amilton cycles in directed graphs.
\newblock {\em European Journal of Combinatorics}, 33(5):750--766, 2012.

\bibitem{papadimitriou1981minimal}
Christos~H Papadimitriou and Mihalis Yannakakis.
\newblock On minimal {E}ulerian graphs.
\newblock {\em Information Processing Letters}, 12(4):203--205, 1981.

\bibitem{schaar1994remarks}
G{\"u}nter Schaar.
\newblock Remarks on {H}amiltonian properties of powers of digraphs.
\newblock {\em Discrete Applied Mathematics}, 51(1-2):181--186, 1994.

\bibitem{schaar1997upper}
G{\"u}nter Schaar and A~Pawel Wojda.
\newblock An upper bound for the {H}amiltonicity exponent of finite digraphs.
\newblock {\em Discrete Mathematics}, 164(1-3):313--316, 1997.

\bibitem{sekanina1960ordering}
Milan Sekanina.
\newblock On an ordering of the set of vertices of a connected graph.
\newblock {\em Publications of the Faculty of Science, University of Brno}, 412:137--142, 1960.

\end{thebibliography}

\end{document}